\documentclass[11pt,letterpaper]{article}
\usepackage[T1]{fontenc}
\usepackage{lmodern,amsmath,amssymb,amsthm,mathtools,booktabs,microtype}
\usepackage[letterpaper,margin=1in]{geometry}
\usepackage[colorlinks=true,allcolors=blue,breaklinks=true]{hyperref}
\hypersetup{pdftitle={Explicit matrix hitting points for noncommutative algebraic branching programs},pdfauthor={Zhipeng Lu}}
\newtheorem{theorem}{Theorem}[section]
\newtheorem{lemma}[theorem]{Lemma}
\newtheorem{proposition}[theorem]{Proposition}
\newtheorem{corollary}[theorem]{Corollary}
\theoremstyle{definition}
\theoremstyle{remark}\newtheorem{remark}[theorem]{Remark}
\DeclareMathOperator{\spanop}{span}
\DeclareMathOperator{\tr}{tr}
\DeclareMathOperator{\im}{im}
\DeclareMathOperator{\rank}{rank}
\newcommand{\field}{\mathbb F}
\newcommand{\Z}{\mathbb Z}\newcommand{\Q}{\mathbb Q}\newcommand{\C}{\mathbb C}
\newcommand{\calM}{\mathcal M}
\newcommand{\eps}{\varepsilon}

\title{Explicit matrix hitting points for\\noncommutative algebraic branching programs}
\author{Zhipeng Lu\\[0.4em]
Shenzhen MSU--BIT University and\\
Guangdong Laboratory of Machine Perception and Intelligent Computing\\
Shenzhen 518172, Guangdong, China\\
\texttt{zhipeng.lu@hotmail.com}}
\date{}
\begin{document}
\hypersetup{pageanchor=false}
\begin{titlepage}
\maketitle
\begin{abstract}
\normalsize
We give a deterministic polynomial-resource black-box identity test for
noncommutative algebraic branching programs (ABPs), using ordinary matrix
evaluations of the output polynomial. Given an alphabet of size $m\ge1$ and an
ABP vertex bound $S\ge1$, we construct one fixed tuple of integer matrices of
order $R=(m+2)d(d+1)/2$, where $d=(m+2)S$, with entry bit length
$O(d\log(d+1))$. Every nonzero polynomial computed by an ABP with at most
$S$ vertices has a nonzero evaluation at this tuple. The construction
accesses neither the hidden ABP nor its intermediate gates. A rectangular
variant improves the order to $O(m^{3/2}S^2)$. Matrix order,
description length, and all returned entries are counted; the total
resources are polynomial. The result holds in characteristic zero and in
characteristic $p>\binom d2$.

The main ingredient is a fixed polynomial frame spanning the joint cyclic
module of every matrix tuple from every prescribed vector. A compensating
block makes a tagged encoding nilpotent of index at most three, while
preserving the original cyclic module. A differential exterior-vector
argument then gives a single triangular recurrence with polynomially many
outputs. Its shared left-linear circuit is converted into actual integer
query matrices by coefficient duality. Two-sided pairing further gives
deterministic nonadaptive reconstruction of an equivalent ABP from
polynomially many output-only matrix queries. The same frame gives
polynomial-cardinality explicit trace separators for simultaneous-conjugation
orbit closures, and a fixed nilpotent substitution for bounded-dimensional
recognizable series. The latter also tests rational formulas whose given
syntax is defined at the scalar origin. We provide complete proofs,
construction costs, and exact implementation checks.
\end{abstract}
\thispagestyle{empty}
\end{titlepage}
\setcounter{page}{1}
\hypersetup{pageanchor=true}

\section{Introduction}
A noncommutative algebraic branching program (ABP) can encode exponentially
many monomials with only a small number of states. Its coefficients,
however, are governed by a finite-dimensional linear representation.
White-box identity tests can exploit that representation; a black-box
algorithm must choose its evaluations without seeing the transition
matrices. This distinction lies behind the gap between the deterministic
polynomial-time white-box test of Raz and Shpilka \cite{RS} and the
quasipolynomial-size black-box hitting set of Forbes and Shpilka
\cite{FSPIT}.

We construct a matrix substitution that works for every ABP of a
prescribed size. Given the number of variables $m$ and an ABP vertex bound
$S$, our algorithm produces a fixed tuple of integer matrices that detects
every nonzero polynomial in the class. The tuple is constructible in polynomial time, and both its description
and the number of returned matrix entries are polynomial in $m,S$. The main idea is to
replace an input-dependent spanning set by a fixed family of polynomial
expressions. A compensating block makes an auxiliary operator nilpotent
without losing the original transitions; differential rigidity then
turns this encoding into a finite recurrence. The recurrence itself
supplies the matrices at which the black box is queried.

\subsection{The hitting point}
We use the matrix black-box model of \cite[Definition~1.4]{FSPIT}:
the algorithm supplies square matrices for the variables and receives the
matrix value of the output polynomial. An ABP is a directed acyclic graph
with a source and a sink, whose edges are labeled by affine linear forms
in $x_1,\ldots,x_m$. Its polynomial is the sum of source--sink path
products, in path order. We measure its state size by the number of
vertices, allowing constants, an empty-word coefficient, and arbitrary
acyclic graphs rather than only layered ones. Parallel edges are merged
by adding their labels, or their complete description is charged to the
input length.

\begin{theorem}[Fixed matrix hitting point]\label{thm:hitting}
Let $m,S\ge1$, and put
\begin{equation}\label{eq:parameters}
r=m+2,\qquad d=rS,\qquad R=r\frac{d(d+1)}2.
\end{equation}
A deterministic algorithm constructs matrices
$T_1,\ldots,T_m\in M_R(\Z)$, depending only on $(m,S)$, such that
\[
f\equiv0\quad\Longleftrightarrow\quad f(T_1,\ldots,T_m)=0
\]
for every polynomial $f$ computed by an ABP with at most $S$ vertices,
over every field of characteristic zero or characteristic
$p>\binom d2$. In positive characteristic the integer entries are reduced
in the field. The matrices are jointly strictly lower triangular.
They can be constructed using $O(mR^2)$ integer operations on numbers of
bit length $O(d\log(d+1))$.

Consequently, in the output-only matrix-evaluation model, deterministic
black-box ncABP identity testing uses polynomial total resources: the
query description has $mR^2$ entries and the answer has $R^2$ entries.
Over $\Q$, the returned entries and evaluation time are polynomially
bounded in $m,S$ and the bit length of the merged edge description.
\end{theorem}

The test constructs the tuple, evaluates $f$ once, and checks every
returned entry for zero. Over a general field, evaluation and zero tests
use exact field operations. Over $\Q$, the bit bound charges the cost of
evaluating any ABP satisfying the size promise, even though that ABP is
not available to the algorithm.

The theorem uses a star-shaped encoding, giving order
$R=O(m^3S^2)$. A balanced rectangular placement of the same data
improves the dependence on the alphabet size.

\begin{corollary}[Rectangular hitting point]\label{cor:rectangle}
Put $a=\lceil\sqrt m\rceil$, $b=\lceil m/a\rceil$,
$r_\square=a+b+1$, $d_\square=r_\square S$, and
$R_\square=r_\square d_\square(d_\square+1)/2$.
Theorem~\ref{thm:hitting} holds with $(d,R)$ replaced by
$(d_\square,R_\square)$, including the construction, entry-height,
and total-resource bounds. Its characteristic condition is
$p>\binom{d_\square}{2}$.
\end{corollary}

Thus the optimized order is $O(m^{3/2}S^2)$; for a fixed alphabet,
both constructions have order quadratic in $S$.
Appendix~\ref{sec:rectangle} gives the construction and its shared-row
implementation, together with the designated frame degree bound
$\lfloor(d_\square-2)/3\rfloor$.

The state bound, rather than the number of monomials, controls the test.
For example, $(x_1+\cdots+x_m)^D$ has a $D+1$-vertex ABP and $m^D$
monomials. Noncommutative polynomial formulas are covered as well: a
formula with $s$ syntax occurrences has an $O(s)$-vertex series--parallel
ABP, obtained by parallel composition for addition and serial composition
for multiplication. This conversion justifies the size promise; the test
still queries only the formula's output. The result concerns matrix
black-box access and state-bounded computations. It does not extend to a
commutative scalar oracle or to arbitrary small arithmetic circuits,
which need not have small ABPs.

\subsection{From hidden states to a fixed query}
The proof separates two tasks: finding expressions that work for every
hidden transition system, and making their information accessible through
ordinary matrix evaluations.

\paragraph{A frame chosen before the matrices are known.}
Write $A_w=A_{i_1}\cdots A_{i_k}$ for $w=i_1\cdots i_k$, and
$A_\eps=I$. Every $S$-vertex affine ABP has a coefficient representation
\begin{equation}\label{eq:intro-realization}
[w]f=\alpha^{\mathsf T}A_w\beta
\end{equation}
on an $S$-dimensional space, including the empty-word coefficient. Hence
$f=0$ exactly when $\alpha^{\mathsf T}$ annihilates the joint cyclic module
\[
\calM(A;\beta)=\spanop_{\field}\{A_w\beta:w\in[m]^*\}.
\]
Finite-dimensionality bounds the size of a spanning set, but does not
choose its elements before $A$ and $\beta$ are known. All word vectors
of length less than $S$ span the module, but their number can grow
exponentially with $S$. We construct
polynomially many polynomials $P_a$, depending only on $m,S$, whose
values $P_a(A)\beta$ span this module for every tuple and every seed.
We call such a family a \emph{fixed frame}. The frame may be overcomplete;
its purpose is to avoid making any input-dependent choice of a basis.

\paragraph{Compensation without information loss.}
Place the hidden matrices in a star with $m+2$ blocks:
\[
K=\begin{pmatrix}0&L\\ C&0\end{pmatrix},\qquad
L=(A_1,\ldots,A_m,-\textstyle\sum_i A_i),\qquad
C=(I,\ldots,I)^{\mathsf T}.
\]
The last block compensates for all the others: $LC=0$, so $K^3=0$.
To retain the individual transitions, assign the blocks distinct tags with
$H=\operatorname{diag}(0,I,\ldots,(m+1)I)$, and let $E_i$ be the
corresponding block projections. Each $E_i$ is a fixed polynomial in
$H$, and, for $1\le i\le m$, $E_0KE_iKE_0$ acts as $A_i$ on the
central block. The compensating block cancels the central return sum;
the tags preserve access to each summand, and hence to every word in
the original matrices.

To exploit this encoding, consider the connection
$\nabla=\partial_z+K/z+H$ and the lifted seed
$\widehat\beta=(\beta,0,\ldots,0)$. The rational span of
$\widehat\beta,\nabla\widehat\beta,\nabla^2\widehat\beta,\ldots$
stabilizes within $d$ vectors. On its own, this only gives a subspace
over $\field(z)$, which may vary with $z$. The rigidity lemma shows
that for the encoded pair it is constant and invariant under both $K$
and $H$. It contains the lifted seed and lies in the scalar extension
of its joint cyclic module, so it equals that scalar extension.
The reason is a degree obstruction. A primitive polynomial exterior
vector records the putative variation. Nilpotence fixes its
constant-coefficient relation, while the zero blocks of $K$ within each
tag eigenspace force its leading coefficient $w_\ell$ to satisfy
$\ell w_\ell=0$. In characteristic zero this rules out every positive
degree; Section~\ref{sec:characteristic} uses a uniform degree bound to
obtain the stated positive-characteristic range.
Taking the coefficients of $z^k\nabla^k\widehat\beta$ now yields a
fixed triangular recurrence. Its central components are the required
$d(d+1)/2$ frame polynomials. Differentiation is used to prove their
spanning property, not as an operation performed by the algorithm.

\paragraph{Turning the frame into an evaluation.}
A spanning family alone does not tell a black box where to evaluate.
The useful additional property is that our recurrence is \emph{left-linear}:
each variable multiplication has the form $x_iP$, with $P$ a previously
computed register. Such a shared circuit has a finite-state coefficient
representation. After eliminating its constant transitions, we obtain
integer transition matrices $T_i$ and a fixed vector $b'$ with
$[w]P_a=e_a^{\mathsf T}T_wb'$, where $a$ indexes the output register.
Consequently,
\begin{equation}\label{eq:intro-duality}
e_a^{\mathsf T}f(T)b'
=\sum_w([w]f)([w]P_a)
=\alpha^{\mathsf T}P_a(A)\beta.
\end{equation}
The coefficient pairing in the middle is therefore a fixed linear
readout of the actual output matrix on the left. If $f(T)=0$, every
frame response is annihilated by $\alpha^{\mathsf T}$. Since the frame
spans the entire hidden cyclic module, all coefficients of $f$ vanish.
This proves the hitting property using only the permitted black-box query.

\subsection{Reconstruction and further consequences}
The frame retains enough information to reconstruct an equivalent ABP,
rather than only test whether its polynomial is zero. Applying it to the transposed
tuple and the output vector spans the observable row space as well as
the reachable column space. Fixed block matrix queries then recover a
finite Hankel table and its one-letter extensions.

\begin{corollary}[Nonadaptive output-only reconstruction]\label{cor:reconstruction}
For the parameters in \eqref{eq:parameters}, put $Q=d(d+1)/2$.
From $(m+1)Q^2$ fixed matrix queries of order $2R$, a deterministic
algorithm reconstructs an equivalent acyclic affine ABP with at most
$S+2$ vertices for every polynomial computed by an $S$-vertex ABP.
All query descriptions, full returned matrices, and field-operation
postprocessing have polynomial total size. The fields are those of
Theorem~\ref{thm:hitting}; over $\Q$, the bit cost is polynomial
in the promised merged edge-description length as well.
\end{corollary}

Every query is fixed before any answer; basis selection takes place only
in the postprocessing. Appendix~\ref{sec:reconstruction} gives the block
queries, including the empty-word readout, and uses standard minimal
Hankel realization \cite[Chapter~2]{BR11} to recover the ABP.
This is a deterministic polynomial-resource, nonadaptive implementation
of reconstruction from output evaluations. For comparison, the exact
learning algorithm of Forbes and Shpilka has membership and equivalence
queries implementable in randomized polynomial time or deterministic
quasipolynomial time \cite[Theorem~1.5]{FSPIT}. The broader learning
framework is developed in \cite{BBBK,KS06}; reconstruction with access
to intermediate gates appears in \cite[Theorem~3.5]{AMS}.

Two further consequences use the same frame. Over $\C$, it gives a
polynomial-cardinality explicit family of trace invariants separating
simultaneous-conjugation orbit closures. The jointly nilpotent query
matrices also test recognizable series of bounded realization dimension.
In characteristic zero, this yields a test for rational formulas whose
given syntax is defined at zero, or at a supplied valid matrix base point. The latter promise
ensures that every inverse gate is defined at the query; the algorithm
does not find the base point. The statements and proofs are in
Appendices~\ref{sec:separators}, \ref{sec:series}, and~\ref{sec:basepoint}.

\subsection{Relation to previous work}
The finite-dimensional viewpoint on noncommutative ABPs goes back to
Nisan's coefficient-rank method \cite{Nisan91}. For matrix black-box
PIT, Bogdanov and Wee gave a randomized test for degree-bounded
noncommutative polynomials \cite[Theorem~4.2]{BW}. Forbes and Shpilka's
deterministic ncABP hitting set has quasipolynomial cardinality and uses
matrices of order $D+1$ for depth $D$ \cite[Corollary~4.7]{FSPIT}.
Our queries have larger, but still polynomial, order; the entire tuple
has polynomial description length, and the returned matrix has
polynomially many entries. This is an
improvement in total explicit resources, rather than in the number of
calls alone: block-diagonal bundling combines any finite hitting set
into one query, but leaves the earlier construction with quasipolynomial
total order.

Automata-based substitutions also underlie the sparse interpolation of
Arvind, Mukhopadhyay, and Srinivasan \cite[Theorem~3.3]{AMS}.
Arvind, Joglekar, Mukhopadhyay, and Raja give randomized sparse tests
polynomial in the variable count and the logarithms of degree and
sparsity \cite[Theorem~2.1 and Corollary~2.2]{AJMR}; Lakhani and
Mukhopadhyay construct a single integer hitting point for
sparsity-bounded polynomials, allowing unrestricted degree in
characteristic zero \cite{LM26}. These results measure monomial count,
whereas our parameter is ABP state size. Lakhani and Saxena study a
different succinct model, giving randomized black-box tests for
constant-depth homogeneous circuits \cite{LS26}.

The rigidity argument is related to degree-zero arguments for invariant
subbundles and constant reductions of differential systems
\cite[Sections~5 and~6.1]{GV}; related small-coefficient criteria appear
in \cite{BG}. Here the compensated encoding enforces nilpotence and
zero same-tag blocks for every transition tuple. Bhardwaj and
Saptharishi also use a formal derivation and a Wronskian argument to
construct scalar hitting sets for commutative polynomials with small
partial derivative spaces \cite{BS26}. That parameter differs from
state size: $x_1\cdots x_t$ has a $2^t$-dimensional commutative partial
derivative space but a $(t+1)$-vertex noncommutative ABP.

The invariant-theoretic consequence builds on classical trace invariants
\cite{Procesi,Razmyslov} and the connection between explicit Noether
normalization and PIT \cite{Mulmuley,FSInv}. It replaces a
quasipolynomial-size fixed separating construction \cite{FSInv} by a
polynomial-size one. In characteristic zero, orbit-closure decision is
already in deterministic NC in the unit-cost arithmetic model
\cite{FSInv}, and input-dependent separators can be found in
polynomial time \cite{DM}.

For rational formulas, linearization connects identity testing with
free-skew-field invertibility \cite{HW}. Deterministic polynomial-time
white-box algorithms over $\Q$ and constructive noncommutative rank
algorithms are known \cite{GGOW,IQS}. Over infinite fields, randomized
polynomial-time matrix black-box testing follows from
\cite[Section~6]{DMbounds}; the general deterministic black-box
construction is quasipolynomial \cite{ACMRIT}. Polynomial-space testing
of shared rational circuits over $\Q$ and finite fields uses explicit
syntax \cite[Theorem~6.2 and Remark~6.4]{AJRIT}. Our rational consequences
instead give fixed queries under the stated domain promise.

Sections~\ref{sec:frame}--\ref{sec:resources} prove the main theorem,
including the characteristic and bit-cost bounds. The proof is
self-contained apart from elementary linear algebra and polynomial
arithmetic; invariant and realization theory enter only in the
consequences.

\section{A fixed frame for every tuple and seed}\label{sec:frame}
We first work over a field $\field$ of characteristic zero. A
left-linear circuit uses additions, scalar multiplications, and operations
$P\mapsto x_iP$. It never multiplies two previously computed
nonconstant polynomials. Circuits may share registers and outputs.

\begin{theorem}[Universal cyclic frame]\label{thm:frame}
Given $m,N\ge1$, put $r=m+2$ and $d=rN$. The integer recurrence
\eqref{eq:component} below defines an indexed family
$\{P_{k,j,0}:0\le j\le k<d\}$ such that, for every $n\le N$,
every tuple $A\in M_n(\field)^m$, and every $v\in\field^n$,
\begin{equation}\label{eq:framespan}
\spanop\{P_{k,j,0}(A)v:0\le j\le k<d\}=\calM(A;v).
\end{equation}
There are $d(d+1)/2$ indexed outputs, of degree at most
$\lfloor(d-1)/3\rfloor$, with a shared left-linear circuit of
$O(md^2)$ gates and integer constants of absolute value at most $d$.
The same conclusion holds in characteristic $p>\binom d2$.
\end{theorem}

The count is of indexed outputs: different indices may represent the same
polynomial or zero. No independence of these polynomials is asserted.
Their evaluated vectors span even when the cyclic module is proper or
the seed is zero. The characteristic-zero proof comes first; the bounded
positive-characteristic extension is proved in
Section~\ref{sec:characteristic}.

\subsection{Nilpotent residues with zero tag blocks}
For an operator $L$ on $E$, let $L^{[a]}$ denote its
\emph{additive} action on $\bigwedge^a E$:
\[
L^{[a]}(u_1\wedge\cdots\wedge u_a)
=\sum_{i=1}^a u_1\wedge\cdots\wedge Lu_i\wedge\cdots\wedge u_a.
\]
If $H$ is diagonalizable over $\field$, write $E_\lambda$ and
$\Pi_\lambda$ for its eigenspaces and their algebraic projections.
No inner product or adjoint is used.

\begin{lemma}[Nilpotent tag rigidity]\label{lem:rigidity}
Let $E$ be finite-dimensional over a characteristic-zero field $\field$.
Suppose $H$ is diagonalizable over $\field$, $K$ is nilpotent,
and $\Pi_\lambda K\Pi_\lambda=0$ for every eigenvalue of $H$.
Every rational subspace $V\subseteq E\otimes_\field\field(z)$ invariant
under $\nabla=\partial_z+K/z+H$ is the scalar extension of a
constant subspace $W\subseteq E$ invariant under both $K$ and $H$.
\end{lemma}
\begin{proof}
The zero subspace is immediate. For $a=\dim_{\field(z)}V>0$, take an
exterior product of a rational basis, clear denominators, and divide all
polynomial coordinates by their gcd. This gives a primitive nonzero vector
\[
w(z)=w_0+w_1z+\cdots+w_\ell z^\ell
\quad\text{in }\bigwedge^aE[z],\qquad w_0,w_\ell\ne0.
\]
Indeed, $w_0=0$ would make $z$ a common coordinate factor. The line it
spans is invariant under the induced connection, so
\begin{equation}\label{eq:exterior}
zw'+(K^{[a]}+zH^{[a]})w=c(z)w
\end{equation}
for some $c\in\field(z)$. Coprime coordinates have a polynomial
Bezout identity, which makes $c$ polynomial. A coordinate of degree
$\ell$ shows that $\deg c\le1$. Write $c=\alpha+\beta z$.
The constant and highest coefficients give
\begin{equation}\label{eq:endcoeff}
K^{[a]}w_0=\alpha w_0,\qquad H^{[a]}w_\ell=\beta w_\ell.
\end{equation}
A nilpotent $K$ is strictly triangular in a suitable basis, and its
additive exterior action is triangular with zero diagonal. Thus
$K^{[a]}$ is nilpotent and $\alpha=0$.

We need a statement about full eigenspace blocks, not merely diagonal
entries. In an eigenbasis of $H$, a nonzero term of $K^{[a]}$ replaces
one factor of tag $\lambda$ by a factor of tag $\mu\ne\lambda$.
It therefore changes the total $H^{[a]}$ weight by $\mu-\lambda\ne0$.
For every eigenvalue $\gamma$ of $H^{[a]}$, its complete spectral
projection $\Pi^{\wedge}_\gamma$ consequently satisfies
\begin{equation}\label{eq:zeroweight}
\Pi^{\wedge}_\gamma K^{[a]}\Pi^{\wedge}_\gamma=0.
\end{equation}
This holds even if different sums of tags coincide.

If $\ell>0$, the coefficient of $z^\ell$ in \eqref{eq:exterior} is
\[
\ell w_\ell+K^{[a]}w_\ell+H^{[a]}w_{\ell-1}
=\beta w_{\ell-1}.
\]
Project onto the $\beta$-eigenspace of $H^{[a]}$. By
\eqref{eq:endcoeff} and \eqref{eq:zeroweight} the second term is killed;
the last terms cancel. Hence $\ell w_\ell=0$, contradicting
characteristic zero. It follows that $\ell=0$.

The constant nonzero decomposable vector $w_0$ defines the subspace by
$u\wedge w_0=0$. These equations have constant coefficients, so
$V=W\otimes_\field\field(z)$ for a constant $W$.
For $u\in W$, invariance says that $Ku+zHu$ vanishes in
$(E/W)[z]$. Comparing coefficients gives $KW,HW\subseteq W$.
\end{proof}

\begin{remark}[Both hypotheses have a role]
Nilpotence alone does not imply rigidity. With
$H=\operatorname{diag}(1,0)$ and
$K=\left(\begin{smallmatrix}1&-1\\1&-1\end{smallmatrix}\right)$,
we have $K^2=0$, but $(1,1+z)^{\mathsf T}$ is horizontal for
$\partial_z+K/z+H$. Its rational line is nonconstant. The same-tag
blocks are nonzero in this example. Conversely, zero same-tag blocks
alone do not suffice. Take $H=\operatorname{diag}(0,1)$ and
$K=\left(\begin{smallmatrix}0&1\\1&0\end{smallmatrix}\right)$.
For $w=(z-1,-1)^{\mathsf T}$ one has $\nabla w=w/z$; its rational
line is nonconstant, and $K$ is not nilpotent. Our encoding below
satisfies both hypotheses identically.
\end{remark}

\subsection{From rigidity to a finite recurrence}
For operators $H,K$ satisfying Lemma~\ref{lem:rigidity}, define
\begin{equation}\label{eq:triangle}
U_{0,0}=I,\qquad
U_{k+1,j}=(K+(j-k)I)U_{k,j}+HU_{k,j-1},
\end{equation}
with an out-of-range index interpreted as zero.

\begin{lemma}[Differential coefficient frame]\label{lem:triangle}
If $\dim E\le d$, then for every $v\in E$, the vectors
$U_{k,j}(H,K)v$, $0\le j\le k<d$, span $\calM(H,K;v)$.
\end{lemma}
\begin{proof}
If $v=0$ there is nothing to prove. Otherwise put $q_i=\nabla^iv$.
Let $t$ be the first dependence index: $q_0,\ldots,q_{t-1}$ are
independent over $\field(z)$ and $q_t$ belongs to their span.
Then $t\le\dim E$. Differentiating a rational linear combination by
$\nabla$ gives derivatives of its coefficients and the next $q_i$.
Thus the first dependence makes their rational span $V$ invariant.
Every $q_i$ belongs to the scalar extension of
$\calM(H,K;v)$. Lemma~\ref{lem:rigidity} makes $V$ constant and
jointly invariant, and it contains $v$. Hence
\[
V=\calM(H,K;v)\otimes_\field\field(z).
\]
Set $f_i=z^iq_i$. Direct calculation gives
\[
f_{i+1}=zf_i'+(K-iI+zH)f_i,
\qquad f_i=\sum_{j=0}^iU_{i,j}(H,K)v\,z^j.
\]
These are polynomial vectors of degree at most $i$. Let $L$ be the
coefficient span of $f_0,\ldots,f_{t-1}$. It is contained in the joint
cyclic module. Conversely, their
rational span is that entire module after extension, and is contained in
$L\otimes_\field\field(z)$. Comparing dimensions gives
$L=\calM(H,K;v)$. Rows through $d-1$ include the first $t$ vectors.
\end{proof}

\subsection{The compensated encoding}
Let $A_1,\ldots,A_m$ act on $E=\field^n$, and set $r=m+2$.
Use blocks numbered $0,1,\ldots,m+1$ on $E^r$. Define
\begin{equation}\label{eq:star}
\begin{split}
H&=\operatorname{diag}(0,I,2I,\ldots,(m+1)I),\\
K(u_0,\ldots,u_{m+1})
&=\left(\sum_{i=1}^mA_iu_i-
  \Bigl(\sum_{i=1}^mA_i\Bigr)u_{m+1},u_0,\ldots,u_0\right).
\end{split}
\end{equation}
Equivalently,
\[
K=\begin{pmatrix}0&L\\C&0\end{pmatrix},\qquad
L=(A_1,\ldots,A_m,-\textstyle\sum_iA_i),\quad C=(I,\ldots,I)^{\mathsf T}.
\]
Since $LC=0$, we have $K^3=0$. All diagonal blocks of $K$ are zero.
For example, for two generators the entire gadget is
\[
K=\begin{pmatrix}
0&A_1&A_2&-A_1-A_2\\
I&0&0&0\\ I&0&0&0\\ I&0&0&0
\end{pmatrix},\qquad H=\operatorname{diag}(0,I,2I,3I).
\]
Although three consecutive $K$ moves vanish, inserting the fixed tag
projections isolates either generator. Nilpotence has therefore not
discarded the original word responses.
The tags are distinct, so the assumptions of the rigidity lemma hold
over $\field$ itself, without any spectral normalization.

For $\widehat v=(v,0,\ldots,0)$,
\begin{equation}\label{eq:modulecopy}
\calM(H,K;\widehat v)=\calM(A;v)^r.
\end{equation}
The right side is invariant and contains the seed. For the reverse
inclusion, each block projection $E_b$ is a polynomial in $H$, by
Lagrange interpolation at the fixed tags. The operator $E_bKE_0$
copies a central vector to leaf $b$, while $E_0KE_iKE_0$ applies
$A_i$ centrally for $1\le i\le m$. Thus all word responses can be
generated centrally and copied to each leaf. These projections are used
only to prove \eqref{eq:modulecopy}; the construction does not compute
them or choose a basis from the input.

\subsection{The actual integer polynomials}
We now compile the lifted recurrence to the original generators. Put
$P_{0,0,0}=1$ and $P_{0,0,b}=0$ for $b>0$. For $0\le j\le k+1$,
define
\begin{equation}\label{eq:component}
\begin{split}
P_{k+1,j,0}
 &=\sum_{i=1}^m x_i(P_{k,j,i}-P_{k,j,m+1})+(j-k)P_{k,j,0},\\
P_{k+1,j,b}
 &=P_{k,j,0}+(j-k)P_{k,j,b}+bP_{k,j-1,b}\qquad(1\le b\le m+1).
\end{split}
\end{equation}
Out-of-range indices are zero. Evaluation of this block vector on $v$
is exactly $U_{k,j}(H,K)\widehat v$ under \eqref{eq:star}.
Take rows $0,\ldots,d-1$ with the uniform bound $d=rN$.
Lemma~\ref{lem:triangle},
\eqref{eq:modulecopy}, and projection onto the central block give
\eqref{eq:framespan}.

Each update has $O(m)$ work for the central component and $O(m)$ work
for all leaves, so the shared left-linear circuit has $O(md^2)$ gates.
All scalars have absolute value at most $d$.
The cancellation in the compensation gives a stronger degree bound
than counting center--leaf moves alone.
\begin{proposition}[Compensated degrees]\label{prop:degree}
Write $C_{k,j}=P_{k,j,0}$, $L_{k,j,b}=P_{k,j,b}$ and
$D_{k,j,i}=L_{k,j,i}-L_{k,j,m+1}$. With $\deg0=-\infty$,
\[
\deg C_{k,j}\le\lfloor k/3\rfloor,\qquad
\deg L_{k,j,b}\le\lfloor(k-1)/3\rfloor,\qquad
\deg D_{k,j,i}\le\lfloor(k-2)/3\rfloor.
\]
\end{proposition}
\begin{proof}
Only $C_{0,0}=1$ is nonzero initially. Subtracting the two leaf
updates cancels their common central term and gives
\[
D_{k+1,j,i}=(j-k)D_{k,j,i}
  +iL_{k,j-1,i}-(m+1)L_{k,j-1,m+1}.
\]
Assuming the three bounds at row $k$, the central update has degree
at most $1+\lfloor(k-2)/3\rfloor=\lfloor(k+1)/3\rfloor$;
the leaf update at most $\lfloor k/3\rfloor$; and the difference
update at most $\lfloor(k-1)/3\rfloor$. These are the next-row bounds.
\end{proof}
The degree bound concerns the free polynomials, including their exact
cancellations. Keeping the differences as separate registers makes
this bound visible to a syntactic degree counter with $O(md^2)$ gates;
the original emitter need not change. This completes the
characteristic-zero proof of Theorem~\ref{thm:frame}.

\section{From a frame to an ordinary matrix query}\label{sec:duality}
The conversion in this section is explicit and does not ask the oracle
for polynomial coefficients. Set $N=S$ and retain every component
register in \eqref{eq:component}, in lexicographic order $(k,j,b)$.
There are exactly $R=rd(d+1)/2$ such registers.

\subsection{Eliminating constant transitions}
Write the register vector as $p$. Its initial vector $b$ is one at
$(0,0,0)$ and zero elsewhere. There are fixed integer matrices
$C_0,C_1,\ldots,C_m$ such that
\begin{equation}\label{eq:automaton}
p=b+C_0p+\sum_{i=1}^m x_iC_ip.
\end{equation}
Every dependency advances the row level $k$, so all $C_i$ are strictly
lower triangular and $C_0^d=0$. Define
\begin{equation}\label{eq:emit}
M=\sum_{a=0}^{d-1}C_0^a=(I-C_0)^{-1},\qquad
T_i=MC_i,\qquad b'=Mb.
\end{equation}
These are ordinary integer matrices and a fixed integer vector.
Multiplication of \eqref{eq:automaton} by $M$ gives
$p=b'+\sum_i x_iT_ip$. Repeated substitution, in word order, yields
\begin{equation}\label{eq:coefficient}
[i_1\cdots i_k]P_a=e_a^{\mathsf T}T_{i_1}\cdots T_{i_k}b'.
\end{equation}
Products of $d$ transitions are zero because each factor advances the
level. Thus the substitution is a finite polynomial expansion.
For every polynomial $f$, it follows that
\begin{equation}\label{eq:pairing}
e_a^{\mathsf T}f(T)b'
=\sum_w([w]f)([w]P_a).
\end{equation}
In particular the order of the $T_i$ agrees with the order of the
letters; no reversal or coefficient oracle is implicit in this identity.

For completeness, the transition entries can be read directly from
\eqref{eq:component}. In a central register $(k+1,j,0)$, the letter
$i$ has coefficient $+1$ from $(k,j,i)$ and $-1$ from
$(k,j,m+1)$. Its constant transition is $(j-k)$ from $(k,j,0)$.
A leaf register $(k+1,j,b)$ has constant transitions $1$ from
$(k,j,0)$, $(j-k)$ from $(k,j,b)$, and $b$ from $(k,j-1,b)$.
Invalid predecessors are omitted. This specifies every matrix entry
from the parameters alone, including the zero registers.

\subsection{The hidden affine ABP representation}
Topologically order the vertices of an ABP. For an edge label
$c_0+\sum_i c_ix_i$ from $u$ to $v$, put its coefficients in
the $(u,v)$ entries of $B_0,B_i$. After merging parallel edges, these
are strictly upper triangular. The finite path sum is
\[
f=e_{\rm src}^{\mathsf T}
\Bigl(I-B_0-\sum_i x_iB_i\Bigr)^{-1}e_{\rm snk}.
\]
The inverse here is a finite sum, because the graph is acyclic.
Put
\begin{equation}\label{eq:abprealization}
Q_0=\sum_{a=0}^{S-1}B_0^a,\qquad
A_i=Q_0B_i,\qquad
\alpha=e_{\rm src},\qquad\beta=Q_0e_{\rm snk}.
\end{equation}
Expanding at the occurrences of nonconstant letters gives
\begin{equation}\label{eq:abpcoefficient}
[i_1\cdots i_k]f
=e_{\rm src}^{\mathsf T}Q_0B_{i_1}Q_0\cdots B_{i_k}Q_0e_{\rm snk}
=\alpha^{\mathsf T}A_{i_1}\cdots A_{i_k}\beta.
\end{equation}
For $k=0$ this is $\alpha^{\mathsf T}\beta$.
The $Q_0$ before the first letter, after the last letter, and between
letters retain all constant paths. Thus \eqref{eq:abpcoefficient}
holds without a homogeneity assumption. A smaller ABP can be padded
with isolated zero states to dimension $S$.

\begin{proof}[Correctness of Theorem~\ref{thm:hitting}]
Assume first characteristic zero. If $f(T)=0$,
\eqref{eq:pairing} and \eqref{eq:abpcoefficient} imply, for every
central frame register $a$,
\[
0=e_a^{\mathsf T}f(T)b'
=\alpha^{\mathsf T}P_a(A)\beta.
\]
Theorem~\ref{thm:frame} applies to the hidden tuple $A$ and its
hidden seed $\beta$. Hence $\alpha^{\mathsf T}$ vanishes on all of
$\calM(A;\beta)$, so \eqref{eq:abpcoefficient} is zero for every
word. Therefore $f\equiv0$. The reverse implication is immediate.
The proof uses the hidden representation only to certify the fixed
query; the algorithm reads solely the returned matrix $f(T)$.
The same argument works over the positive-characteristic fields
covered by the frame theorem, proved next.
\end{proof}

\section{Characteristic and construction costs}\label{sec:resources}
\subsection{The bounded positive-characteristic argument}\label{sec:characteristic}
Lemma~\ref{lem:rigidity} was stated in characteristic zero. It is
false for arbitrary rational subspaces in characteristic $p$:
with $K=H=0$, the line generated by $e_1+z^pe_2$ is invariant
and nonconstant. For the cyclic spaces actually used by the algorithm,
however, the primitive exterior degree has a uniform bound.

\begin{proposition}[Bounded-characteristic cyclic frame]\label{prop:characteristic}
The frame in Theorem~\ref{thm:frame} spans the same joint cyclic
module in characteristic $p>\binom d2$.
\end{proposition}
\begin{proof}
Suppose $p>\binom d2$ and use the compensated encoding of a
tuple of dimension at most $N$. If the seed is nonzero, let
$t\le d$ be the first rational Krylov dependence index, as in
Lemma~\ref{lem:triangle}. The independent polynomial vectors
$f_i=z^i\nabla^iv$, $0\le i<t$, have degree at most $i$.
Their nonzero exterior product has degree at most
\[
\sum_{i=0}^{t-1}i=\binom t2\le\binom d2.
\]
Division by its coordinate gcd only lowers degree. Thus its primitive
representative has $0\le\ell<p$.
Since $d=(m+2)N\ge3$, the same characteristic bound makes the
tags $0,\ldots,m+1$ distinct. The residue is nilpotent, and its
additive exterior action still has zero same-weight blocks: one
nonzero term changes exactly one tag by a nonzero difference in the
field. The proof of Lemma~\ref{lem:rigidity} therefore gives
$\ell w_\ell=0$. Now $\ell<p$ forces $\ell=0$.
The constant-space and coefficient-spanning arguments follow unchanged.
This proves the characteristic extension of Theorem~\ref{thm:frame},
and completes the algebraic assertion of Theorem~\ref{thm:hitting}.
\end{proof}

The bound is a sufficient characteristic condition, not a claim of
necessity or coverage of all smaller fields. The construction actually
fails for some small characteristics: the supplied exact implementation
at $(m,S)=(2,3)$, reduced modulo $2$, sends the nonzero three-vertex
ABP polynomial $x_1^2$ to the full zero matrix of order $312$.
No positive-characteristic invariant-ring consequence is used.

\subsection{Explicit entries and their heights}
Each row of $C_0$ has at most three predecessors, and each row of a
letter matrix $C_i$ has at most two entries. The row-sum norm satisfies
\[
\|C_0\|_\infty\le3d,\qquad
\|C_i\|_\infty\le2,\qquad \|b\|_\infty=1.
\]
Consequently every entry of $T_i$ and $b'$ has absolute value at most
\begin{equation}\label{eq:height}
2d(3d)^{d-1}.
\end{equation}
Indeed, \eqref{eq:emit} is a sum of $d$ powers and the norm is
submultiplicative. The stated bit bound follows directly.

There is no need to multiply dense matrices to obtain this finite
inverse. For $T_i$, solve $T_i=C_i+C_0T_i$ row by row in the
register order. Start its current row with the corresponding row of
$C_i$, then add the at most three already computed predecessor rows,
with the coefficients from $C_0$. This uses $O(R)$ bounded-height
integer operations per row, or $O(mR^2)$ for all matrices. The same
triangular recurrence computes $b'$. Intermediate partial rows satisfy
the same absolute-sum bound as \eqref{eq:height}. The construction is
therefore polynomial in bit time as well as in arithmetic operations.

\subsection{Charging the oracle and its returned matrix}
The dense query description has $mR^2$ integer entries, each of
$O(d\log(d+1))$ bits; a sparse encoding may save space but is not
needed for the polynomial bound. The response has $R^2$ entries,
and its zero test charges all of them. These costs explain why the
one-query statement does not express information compression.

For a rational ABP, let $L$ be the total bit length of its merged
edge coefficients. A common denominator $q$ has $O(L)$ bits.
Every path has at most $S-1$ edges, so multiplication by $q^{S-1}$
clears all denominators in its matrix evaluation. The number of paths
in a simple acyclic graph is at most $2^S$; expanding the affine labels
adds at most $(m+1)^{S-1}$ terms per path. If the query entries have
$B=O(d\log(d+1))$ bits, a product of $k\ge1$ query matrices has
entry heights bounded by $R^{k-1}2^{kB}$. For $k=0$, the empty
product is $I_R$ and its entries have absolute value at most $1$.
Combining these estimates gives a polynomial bit bound in $m,S,L$
for all returned entries. Topological dynamic programming evaluates the
ABP with polynomially many rational $R\times R$ matrix operations,
whose intermediate values obey the same path bounds. This proves the
rational bit-cost assertion without imposing a hidden unit-cost
assumption on long integers or parallel-edge multiplicities.

\section{Implementation and consequence boundaries}\label{sec:verification}
The accompanying standard-library Python emitter materializes the
integer matrices directly from the register recurrence, without expanding
the probe monomials. For $(m,S)=(2,2)$ it produces order $144$,
$924$ nonzero entries in each matrix, and maximum entry bit length $12$.
For $(2,4)$ the order is $544$. These are reproducible construction
examples, not matrix-order optimality or practical performance claims.

Exact checks use independently expanded affine source--sink paths,
full matrix evaluation, and the coefficient-dual readout, including
constant coefficients, cancellations, proper cyclic modules, and allowed
positive characteristics. The independent run checks eight affine ABP
fixtures over $\Q$ and $\mathbb F_{127}$: $16\,688$ compiler
coefficient equalities, $328$ affine-realization coefficient equalities,
$5\,968$ component readouts, and $16$ full matrix comparisons.
It separately reproduces the characteristic-two failure.
The verification records specify each finite input and its result.
The general theorems follow from the proofs above, not from sampling.
The Lean~4 \cite{Lean4} artifact, built on mathlib \cite{Mathlib},
verifies eigenbasis transport of rigidity and the actual
star and rectangular hitting equivalences in both characteristic regimes,
as well as recognizable-series finite substitution. Its statement map and
kernel records also cover parallel-edge merging, two-sided series readouts,
deterministic minimal recovery, nilpotent flags, construction and recovery
arithmetic models, canonical rational answer heights, and rational identity
testing at valid supplied points. Scalar invariant polynomials, homogeneous
components and shared circuits are formalized. Full word-trace generation,
actual orbit-closure geometry and integrality over the specified homogeneous
subalgebra also have kernel proofs; native interpreter and integer-library
runtimes are not certified.
AI tools assisted with intermediate proof development, literature review,
exact implementation checks, and Lean verification; their role is
disclosed at the end.

The rectangular emitter is checked against a naive triangular solve and
explicit component expansions, including its shared-row implementation.
The reconstruction checks execute $9\,486$ complete fixed
block queries across nine field/fixture cases, recover minimal
representations of dimensions zero through three, and compare the
returned acyclic ABPs with literal path expansions. Thirty additional
commutator queries check noncommutative word order. A supplied-base-point
example checks both inverse identities for a full $936$-dimensional
perturbation, using its independently proved two-state realization bound.

The accompanying repository at \url{https://github.com/Lzp88/PIT}
contains the Lean~4 audit sources, the statement map and kernel audit
records, and the exact Python reproduction scripts.

\clearpage
\appendix
\section{Fixed separating invariants}\label{sec:separators}\label{app:consequences}

The compensated frame also gives a fixed polynomial-size map for the
simultaneous-conjugation quotient.  The output is a family of scalar
invariants, each obtained by taking the trace of a matrix polynomial.

\begin{theorem}[Fixed orbit-closure separators]\label{thm:separators}
Fix $m,N\ge1$ and put $\Delta=2(m+2)N^2$.  There is an indexed family
of at most $\Delta(\Delta+1)/2$ trace polynomials, constructed from
$(m,N)$ alone, with the following property.  For every fixed $n\le N$
and $A,B\in M_n(\C)^m$, its values on $A$ and $B$ agree if and only if
their simultaneous-conjugation orbit closures intersect.  The degree is
at most $\lfloor(\Delta-1)/3\rfloor$, and the matrix-polynomial outputs have a shared
$O(m\Delta^2)$ arithmetic circuit acting on matrices of the original
order.  A direct expansion into scalar entries has
$O(mn^3\Delta^2)$ arithmetic gates, including the final trace readouts.
\end{theorem}

\begin{proof}
Construct the central frame of Theorem~\ref{thm:frame} with module
dimension bound $2N^2$, and take
\[
 \mathcal S_{m,N,n}
 =\bigl(\tr P_{k,j,0}(A):0\le j\le k<\Delta\bigr).
\]
Each coordinate is a polynomial invariant: evaluation of a
noncommutative polynomial commutes with simultaneous conjugation, and
trace is unchanged by conjugation.  The notation denotes indexed
outputs; zeros and repeated polynomials cause no difficulty.

Suppose their values agree on $A$ and $B$.  On
$E=M_n(\C)\oplus M_n(\C)$, define
\[
 L_i(X,Y)=(A_iX,B_iY),\qquad
 v=(I_n,I_n),\qquad
 \ell(X,Y)=\tr X-\tr Y.
\]
The dimension of $E$ is at most $2N^2$.  Equality of the displayed
coordinates means that $\ell$ annihilates the central frame responses
$P_{k,j,0}(L)v$.  Theorem~\ref{thm:frame} makes these responses span
the full joint cyclic module.  Therefore, for every word $w$,
\[
 0=\ell(w(L)v)=\tr w(A)-\tr w(B).
\]
In characteristic zero the word traces generate the
simultaneous-conjugation invariant ring \cite[Theorem~1.3]{Procesi}.
Over $\C$, equality of all invariants is equivalent to intersection
of the orbit closures \cite[Theorem~1.1]{DM}; see also
\cite[Section~1.2]{FSInv}. Conversely, each displayed
polynomial is invariant and hence constant on an orbit closure, so an
intersection forces equality of all coordinates.

The number of indices and shared matrix circuit size follow from
\eqref{eq:component}. Proposition~\ref{prop:degree} gives
$\deg P_{k,j,0}\le\lfloor k/3\rfloor$, including the compensated
cancellation, and hence
\[
 \deg\mathcal S_{m,N,n}\le
 \left\lfloor\frac{\Delta-1}{3}\right\rfloor.
\]
Replacing an $n\times n$ matrix multiplication by its elementary
scalar circuit costs $O(n^3)$ gates.  Reading all final traces costs
$O(n\Delta^2)$ more, which is within the stated bound.
\end{proof}

For inputs with orders at most $N$, the constant output
$\tr P_{0,0,0}=\tr I$ first distinguishes unequal orders.  When the
orders agree, Theorem~\ref{thm:separators} applies in the common
representation space.  No trace value is fed back into the matrix
recurrence.

\begin{corollary}[An integral separating subalgebra]\label{cor:integral}
For each fixed $n\le N$, there is a polynomial-size family of
homogeneous invariants, with polynomial-size shared circuits, such
that $\C[M_n^m]^{\mathrm{GL}_n}$ is integral over the subalgebra they
generate.  Their degrees are at most $\lfloor(\Delta-1)/3\rfloor$.
\end{corollary}

\begin{proof}
Split each coordinate of $\mathcal S_{m,N,n}$ into homogeneous
components in the scalar matrix entries.  The action of
$\mathrm{GL}_n$ preserves degree, so each component is invariant.
Agreement of all components implies agreement of the original
coordinates and hence of every invariant.  Thus the enlarged family
is separating.  Its constant components have the same value at every
point of the fixed-order space and may be removed.

The recurrence is left-linear: variable multiplication raises the
degree by one, and its remaining operations preserve degree.  Carrying
a degree index through its registers therefore computes all components
with an additional factor $O(\Delta)$ in circuit size and output
count.  This remains polynomial.  The invariant ring is integral over
the subalgebra generated by a finite homogeneous separating family
\cite[Theorem~2.3.12]{DK}; the form used here is also stated in
\cite[Theorem~3.9]{FSInv}. This proves the claim.
\end{proof}

\begin{remark}
The conclusion concerns orbit closures and does not assert
simultaneous similarity.  For example, a nonzero nilpotent Jordan block
and the zero matrix of the same order have equal positive-length word
traces but are not similar; their orbit closures intersect.  The
construction makes no separating-degree improvement over the bounds
of \cite{DM}.  Corollary~\ref{cor:integral} asserts integrality, without
asserting algebraic independence, a minimal homogeneous system of
parameters, or generation of the full invariant ring.
\end{remark}

\section{Recognizable series and formulas defined at zero}
\label{sec:series}

Recognizable series belong to the classical theory of weighted automata
\cite{Schutzenberger61,BR11}. Matrix-substitution criteria for common
coefficient kernels, observability, and controllability appear in
\cite{AK04}. Here the substitution is fixed from the alphabet and
realization bounds alone, with polynomial total explicit resources.

A formal noncommutative series $F\in
\field\langle\!\langle x_1,\ldots,x_m\rangle\!\rangle$ is
\emph{recognizable of realization dimension at most $N$} if there are
$q\le N$, matrices $A_i\in M_q(\field)$, and vectors
$\alpha,\beta\in\field^q$ such that
\[
 [w]F=\alpha^{\mathsf T}A_w\beta
 \quad\text{for every word }w.
\]
The empty word is included, with $A_{\varepsilon}=I_q$.

\begin{theorem}[A fixed nilpotent query for recognizable series]
\label{thm:series}
Let $m,N\ge1$ and $d=(m+2)N$, and suppose $\field$ has characteristic zero or
characteristic $p>\binom d2$.  Let $T=T_{m,N}$ be the integer tuple
of Theorem~\ref{thm:hitting}, interpreted over $\field$.  For every
recognizable series $F$ of realization dimension at most $N$, the
finite substitution
\[
 F(T)=\sum_{|w|<d}[w]F\,T_w
\]
is defined and satisfies $F(T)=0$ if and only if $F$ is the zero series.
\end{theorem}

\begin{proof}
The common level filtration of the transition matrices gives
$T_w=0$ for $|w|\ge d$.  Consequently substitution in an arbitrary
formal series is a finite sum.  For each central frame output $P_a$,
the coefficient-dual identity from \eqref{eq:automaton} gives, using
its dual seed $b'$,
\[
 \begin{aligned}
 e_a^{\mathsf T}F(T)b'
 &=\sum_w([w]F)([w]P_a)\\
 &=\alpha^{\mathsf T}P_a(A)\beta.
 \end{aligned}
\]
These sums are finite because every $P_a$ has degree less than $d$.
If $F(T)=0$, then $\alpha^{\mathsf T}$ annihilates all the central
frame responses.  Theorem~\ref{thm:frame}, in its stated
characteristic range, makes them span the joint cyclic module from
$\beta$.  Thus $\alpha^{\mathsf T}A_w\beta=0$ for every word $w$.
The converse is immediate.
\end{proof}

One may equivalently express the finite substitution by
\[
 F(T)=(\alpha^{\mathsf T}\otimes I)
 \left(I-\sum_{i=1}^m A_i\otimes T_i\right)^{-1}
 (\beta\otimes I).
\]
The inverse is a finite geometric sum, since every product of $d$
transitions is zero.  This identity explains the substitution; the
black-box procedure does not obtain the hidden realization.  An
abstract formal series need not supply a matrix-evaluation oracle.
The algorithmic consequence requires an oracle allowing this jointly
nilpotent substitution.

\begin{corollary}[Rational formulas syntactically defined at zero]
\label{cor:rational}
Let $m\ge1$, let $\field$ have characteristic zero, and let $q$ be a
scalar-coefficient noncommutative rational formula whose syntax is
defined at the scalar origin.  Thus every operand of an inverse gate
has a nonzero scalar value there.  Put
\[
 \kappa(q)=\#\text{constant occurrences}
       +2\#\text{letter occurrences}
       +\#\text{inverse occurrences}.
\]
For every $N\ge\kappa(q)$, the tuple $T_{m,N}$ is in the syntactic
domain of $q$, and $q(T_{m,N})=0$ if and only if $q$ is a rational
identity.
\end{corollary}

\begin{proof}
Induct through the formula.  Every subexpression evaluated at $T$
is its scalar value at zero times $I$, plus a strictly lower
triangular matrix.  Addition and multiplication preserve this form.
At an inverse gate the scalar value is nonzero, so the evaluated
operand is invertible; the inverse is obtained by a finite geometric
sum.  The same induction shows agreement with substitution in the
formal expansion at zero.

That expansion has a realization of dimension at most $\kappa(q)$
\cite[Corollary~5.1 and Remark~5.2]{Volcic}.
Theorem~\ref{thm:series} therefore gives that $q(T)=0$ if and only
if this expansion is zero.  To see explicitly that a zero expansion
is a rational identity, take generic matrices $Z_i$ of any fixed
order and a central variable $t$.  The rational matrix $q(tZ)$ has
denominators nonzero at $t=0$; its expansion there is zero.  The
embedding of the corresponding commutative rational functions into
Laurent series at zero is injective, so $q(tZ)=0$.  Rescaling the
generic entries yields $q(Z)=0$.  This holds at every order and thus
on the syntactic matrix domain of $q$.  Conversely a rational
identity vanishes at the supplied query, which is in that domain.
\end{proof}

With formula size counting all syntax occurrences, $\kappa(q)$ is
linear in size.  This yields a polynomial-dimensional fixed query
in the exact field-arithmetic model.  A bit-cost statement additionally
requires bounds on coefficient encoding and oracle evaluation cost.
The promise is about the given syntax: regularity of the represented
function at zero alone does not ensure that its particular expression
is defined there.  For instance, $x_1^{-1}$ is undefined at every
supplied nilpotent tuple, and the inverse of a commutator has no
scalar point in its domain.  The corollary supplies no base-point
construction for unrestricted rational formulas.  General black-box
rational identity testing over infinite fields has randomized polynomial-time
\cite{DMbounds} and deterministic quasipolynomial-time \cite{ACMRIT}
algorithms; the fixed deterministic query here uses the origin-domain
promise.

\section{Rectangular compensation and alphabet dependence}\label{sec:rectangle}
The star encoding makes the core proof easy to follow. A rectangular
placement of the generators gives a smaller hitting order when the
alphabet grows. We retain the star construction and specify this
alternative separately, including the cost of its denser constant rows.

For $m,N\ge1$, put
\[
a=\lceil\sqrt m\rceil,\qquad b=\lceil m/a\rceil,\qquad
r_\square=a+b+1,\qquad d_\square=r_\square N.
\]
Let $\iota(c,q)=cb+q+1$ for $0\le c<a$, $0\le q<b$,
and set $A_{cq}=A_{\iota(c,q)}$ when $\iota(c,q)\le m$, and
$A_{cq}=0$ otherwise. On $E^{r_\square}$ use $a$ center blocks
numbered $0,\ldots,a-1$ and $b+1$ leaf blocks numbered
$a,\ldots,a+b$. Define
\begin{equation}\label{eq:rectangle-operators}
K=\begin{pmatrix}0&L\\ C&0\end{pmatrix},\qquad
L_{c,q}=A_{cq}\ (q<b),\qquad L_{c,b}=-\sum_{q<b}A_{cq},
\end{equation}
where every block of the $(b+1)\times a$ matrix $C$ is $I_E$.
Let $H=\operatorname{diag}(0I_E,I_E,\ldots,(r_\square-1)I_E)$.
The lifted seed is $\widehat v=(v,0,\ldots,0)$ in center block zero;
the designated output is leaf block $a$.

\begin{theorem}[Rectangular matrix hitting point]\label{thm:rectangle}
For this encoding, recurrence~\eqref{eq:triangle} through row
$d_\square-1$ gives a fixed complete cyclic frame when projected onto
any leaf block. In particular its designated leaf polynomials are
$Q_\square=d_\square(d_\square+1)/2$ indexed outputs, of degree at most
$\lfloor(d_\square-2)/3\rfloor$, with a shared left-linear circuit
of $O(md_\square^2)$ gates.

Taking $N=S$, their coefficient-dual construction gives fixed integer
matrices of order
\[
R_\square=r_\square Q_\square
 =O(m^{3/2}S^2)
\]
which hit all affine ncABPs with at most $S$ vertices over characteristic
zero, and over characteristic $p>\binom{d_\square}{2}$.
They are jointly strictly lower triangular, their entries have
$O(d_\square\log(d_\square+1))$ bits, and their explicit construction
uses $O(mR_\square^2)$ bounded-height integer operations. The complete
query description and returned matrix have $mR_\square^2$ and
$R_\square^2$ entries, respectively. The rational bit-cost conclusion
of Theorem~\ref{thm:hitting} holds with these parameters as well.
\end{theorem}
\begin{proof}
Each block of $LC$ is a row sum of $L$ and is zero. Thus $K^3=0$.
Every equal-tag block of $K$ is zero. The tag-rigidity argument therefore
applies in characteristic zero; its bounded-characteristic cyclic-space
proof in Section~\ref{sec:characteristic} also applies with
$d_\square$ in place of $d$. In particular the characteristic bound
keeps all integer tags distinct and bounds the primitive exterior degree
strictly below $p$.

Write $M=\calM(A;v)$ and let $W=\calM(H,K;\widehat v)$.
We have $W\subseteq M^{r_\square}$ since the latter is invariant
and contains the seed. Let $E_t$ be the projection onto block $t$;
it is a polynomial in $H$ by interpolation at its distinct fixed tags.
For every leaf $t$, the vector $E_tKE_0\widehat v$ is $v$ supported
in leaf $t$. If $A_wv$ is available separately in every leaf and
$A_i=A_{cq}$, then $E_cKE_{a+q}$ applies $A_i$ to that vector
and puts the result in center $c$. Applying $E_tKE_c$ copies the
result into any prescribed leaf $t$. Induction over words consequently
shows
\begin{equation}\label{eq:rectangle-projection}
\text{projection of }W\text{ onto leaf }t=M\qquad(a\le t<a+b+1).
\end{equation}
The induction builds each word from its last letter toward its first,
as required for $A_wv$. No equality $W=M^{r_\square}$ is needed.
Lemma~\ref{lem:triangle} now makes the coefficients of the first
$d_\square$ differential Krylov vectors span $W$, and their leaf
projections span $M$.

Here is the polynomial recurrence which also specifies all query entries.
Use center components $C_{k,j,c}$ and leaf components $L_{k,j,q}$,
where $0\le c<a$ and $0\le q\le b$. Put $C_{0,0,0}=1$ and
all other row-zero components equal to zero. Terms outside their
triangular ranges are zero. For every subsequent row define
\begin{align}\label{eq:rectangle-components}
C_{k+1,j,c}&=
 \sum_{\substack{0\le q<b\\\iota(c,q)\le m}}
 x_{\iota(c,q)}(L_{k,j,q}-L_{k,j,b})
 +(j-k)C_{k,j,c}+cC_{k,j-1,c},\notag\\
L_{k+1,j,q}&=\sum_{c=0}^{a-1}C_{k,j,c}
 +(j-k)L_{k,j,q}+(a+q)L_{k,j-1,q}.
\end{align}
The center tag term $cC_{k,j-1,c}$ includes the nonzero center tags;
it is essential when $a>1$. There are $m$ data slots in total.
Sharing the sum of center components across the leaf updates gives
$O(m)$ gates per $(k,j)$ cluster, since $r_\square=O(m)$.
This proves the circuit bound.

For the degree bound put $D_{k,j,q}=L_{k,j,q}-L_{k,j,b}$ for
$q<b$. The common center sum cancels and gives
\[
D_{k+1,j,q}=(j-k)D_{k,j,q}
 +(a+q)L_{k,j-1,q}-(a+b)L_{k,j-1,b}.
\]
An induction, starting from the given row-zero components and using
$\deg0=-\infty$, yields simultaneously
\[
\deg C_{k,j,c}\le\lfloor k/3\rfloor,\qquad
\deg L_{k,j,q}\le\lfloor(k-1)/3\rfloor,\qquad
\deg D_{k,j,q}\le\lfloor(k-2)/3\rfloor.
\]
Indeed the only degree-increasing term in a center update is $x_iD$,
whose bound is $1+\lfloor(k-2)/3\rfloor=\lfloor(k+1)/3\rfloor$;
a leaf update has bound $\lfloor k/3\rfloor$, and a difference update
has bound $\lfloor(k-1)/3\rfloor$. A left-linear circuit that stores
the differences as separate registers certifies these bounds
syntactically with the same $O(md_\square^2)$ gate count. The emitter
may retain the original components: the sharper degree concerns their
polynomials, independently of a degree counter on an unsimplified DAG.

Retain every component register in lexicographic order $(k,j,t)$.
As in~\eqref{eq:automaton}--\eqref{eq:emit}, the constant matrix
$C_0$ and the letter matrices $C_i$ advance row level, so
$C_0^{d_\square}=0$ and
\[
T_i=(I-C_0)^{-1}C_i,\qquad \eta=(I-C_0)^{-1}b_0
\]
are fixed integer objects. Here $b_0$ is the initial register vector.
A constant center row has at most two predecessors. A constant leaf row
has the $a$ center predecessors together with its two tag and diagonal
predecessors. Nevertheless
\[
\|C_0\|_\infty\le a+2d_\square\le3d_\square,
\qquad \|C_i\|_\infty\le2.
\]
The finite geometric-sum argument bounds every query or seed entry by
$2d_\square(3d_\square)^{d_\square-1}$, with the stated bit length.

For clarity, direct separate addition of the $a$ center predecessor rows
for each leaf would not give the asserted construction count.
Instead, for each letter matrix being materialized and each $(k,j)$,
first form the single row vector
\[
\sigma_{k,j}=\sum_{c<a}(T_i)_{(k,j,c),*}.
\]
Each leaf in the next-row cluster uses this same vector, together with
its at most two own predecessors. Forming the shared sum and all center
and leaf rows costs $O(r_\square R_\square)$ integer operations per
cluster. There are $Q_\square$ clusters. Thus each letter costs
$O(r_\square Q_\square R_\square)=O(R_\square^2)$, and all letters
cost $O(mR_\square^2)$. The seed uses the analogous scalar recurrence.
The absolute-sum height argument also bounds intermediate shared sums
and partial row additions by the same order of height.

Coefficient duality~\eqref{eq:coefficient}--\eqref{eq:pairing} applies
to these registers. The designated leaf polynomial in row one is $1$,
so it retains the empty-word coefficient. Using
\eqref{eq:rectangle-projection} for the hidden affine ABP realization
now proves the hitting statement exactly as for the star. The rational
path-height accounting uses only polynomial matrix order and query-entry
height, and therefore carries over unchanged.
\end{proof}

For example, at $(m,S)=(4,2)$ the rectangular order is $275$,
whereas the star order is $468$. At $(5,2)$ the respective orders are
$468$ and $735$. These are improvements in the alphabet parameter;
for every fixed alphabet both constructions retain quadratic order in $S$.
The accompanying rectangular emitter uses the shared-center-row algorithm.
Independent exact tests cover $56$ arbitrary lifted-frame cases over
$\Q$ and allowed prime fields, all $176$ leaf projections in those
cases, $99\,083$ scalar coefficient equalities, $1\,022$ component and
$347$ difference degree checks, and $24$ comparisons with a naive full
triangular solve. A complete commutator evaluation at order $1050$
is also nonzero. These are finite checks, distinct from the proof and
from Lean verification.

\begin{proposition}[Elementary single-query order lower bound]\label{prop:query-lower}
Let $m,S\ge1$ and let $\field$ be the coefficient field. Any one
fixed tuple $T_1,\ldots,T_m\in M_R(\field)$ hitting every affine
ncABP over $\field$ with at most $S$ vertices must have $R\ge S$.
\end{proposition}
\begin{proof}
At order zero, the constant polynomial $1$ already evaluates to the
zero matrix, so assume $R\ge1$.
If its order is $R<S$, the first matrix has monic characteristic polynomial
$h(t)=\sum_{i=0}^R h_it^i\in\field[t]$, which annihilates it.
This polynomial has an affine ABP on $R+1$ vertices: on vertices
$0,\ldots,R$, use the chain edges $i\to i+1$ labeled $x_1$ for
$0\le i<R-1$, the final edge $(R-1)\to R$ labeled
$h_Rx_1+h_{R-1}$, and shortcuts $i\to R$ labeled $h_i$ for
$0\le i<R-1$. The source is zero and the sink is $R$.
Its path sum is the nonzero polynomial $h(x_1)$, which vanishes at
the proposed tuple. Since $R+1\le S$, this contradicts the hitting property.
\end{proof}
This lower bound does not match the quadratic dependence of our upper
bound and entails no optimality claim. It concerns queries over the
coefficient field. Extension-field queries can have smaller order:
the scalar $[\sqrt2]$ hits every nonzero rational-coefficient affine
polynomial in one letter, although such polynomials have two-vertex ABPs.

\section{Nonadaptive output-only reconstruction}\label{app:reconstruction}\label{sec:reconstruction}

The fixed frame also supplies complete finite Hankel data through ordinary
matrix evaluations.  Throughout this appendix we use the star frame of
Theorem~\ref{thm:frame}, with
\[
 r=m+2,\quad d=rS,\quad Q=d(d+1)/2,\quad R=rQ.
\]
Index its central polynomials by $P_a$, $0\le a<Q$, with $P_0=1$.
Write $\rho(a)$ for the corresponding component-register index, and put
$\eta=b'$.  Equation~\ref{eq:coefficient} states
\[
 [w]P_a=e_{\rho(a)}^{\mathsf T}T_w\eta.
\]
For $P$, let $P^{\mathrm{rev}}$ reverse each word, retaining its scalar
coefficient.  Thus $P^{\mathrm{rev}}(A)^{\mathsf T}=P(A^{\mathsf T})$;
the transpose here involves no conjugation.

\begin{theorem}[Fixed-query ABP reconstruction]\label{thm:reconstruction}
Let $m,S\ge1$.  Over every field of characteristic zero or characteristic
$p>\binom d2$, an unknown polynomial computed by an affine ncABP with at
most $S$ vertices can be reconstructed deterministically from output-only
matrix evaluations.  All queries are fixed by $(m,S)$ before any answer
is received.  The algorithm returns a minimal linear representation and
an equivalent acyclic ABP with at most $S+2$ vertices.

There are $(m+1)Q^2$ queries of order $2R$, with integer entries of
$O(d\log(d+1))$ bits.  Query generation, the complete returned data, and
postprocessing use polynomial resources.  Over $\Q$, bit complexity is
polynomial in these parameters and the merged edge-description length of
the promised ABP.
\end{theorem}

\subsection{The block queries, including their constant terms}
For $a,b<Q$, put $E_{ab}=e_{\rho(a)}e_{\rho(b)}^{\mathsf T}$ and define
\begin{equation}\label{eq:reconstruction-queries}
 X_i^{a,b}=\begin{pmatrix}T_i^{\mathsf T}&E_{ab}T_i\\0&T_i\end{pmatrix},
 \qquad
 Y_i^{a,b,j}=\begin{pmatrix}T_i^{\mathsf T}&\mathbf1_{i=j}E_{ab}\\0&T_i\end{pmatrix}.
\end{equation}
Use the fixed readout vectors
\[
 u^{\mathsf T}=(\eta^{\mathsf T},0),\quad
 v_{ab}=\binom{E_{ab}\eta}{\eta},\quad v_0=\binom0\eta.
\]
Every matrix and vector in these expressions is known to the algorithm.
It transposes its own query matrices, without requesting an operation on
the unknown polynomial or its hidden representation.

\begin{lemma}[Two-sided data from actual evaluations]\label{lem:reconstruction-queries}
For every polynomial $f$, the full output matrices satisfy
\begin{align}
 u^{\mathsf T}f(X^{a,b})v_{ab}
 &=\sum_{w=st}([w]f)([s]P_a^{\mathrm{rev}})([t]P_b),\label{eq:reconstruction-g}\\
 u^{\mathsf T}f(Y^{a,b,j})v_0
 &=\sum_{w=sjt}([w]f)([s]P_a^{\mathrm{rev}})([t]P_b).\label{eq:reconstruction-gi}
\end{align}
The first sum includes empty prefixes, empty suffixes, and the empty word.
These identities also hold for recognizable series with nilpotent
matrix-evaluation access.
\end{lemma}
\begin{proof}
In the upper-right block of $X_w$, a product crosses from the upper block
to the lower one exactly once.  A crossing at a letter uses $E_{ab}T_i$,
so that letter is the first letter of the suffix.  The additional upper
component $E_{ab}\eta$ in $v_{ab}$ supplies the suffix-empty term.
Consequently the contraction is
\[
 \sum_{w=st}([w]f)\eta^{\mathsf T}(T^{\mathsf T})_s E_{ab}T_t\eta.
\]
For the word-order product of transposes,
$(T^{\mathsf T})_s=(T_{\mathrm{rev}(s)})^{\mathsf T}$.  The coefficient
identity therefore makes its two factors $[s]P_a^{\mathrm{rev}}$ and
$[t]P_b$.  At $w=\eps$ the readout is
$\eta^{\mathsf T}E_{ab}\eta$, which gives precisely the constant term
on the right.  This proves~\eqref{eq:reconstruction-g}.

For $Y$, the unique crossing reads the distinguished letter $j$ and
contributes $E_{ab}$, giving~\eqref{eq:reconstruction-gi}.  Its
empty-word readout is zero.  The diagonal block tuples have every
length-$d$ product zero, and a product can cross blocks at most once.
Every product of $2d$ query matrices is consequently zero.  Formal-series
substitution and both coefficient sums are thus finite.
\end{proof}

\subsection{Minimal realization and an acyclic output}
\begin{proof}[Proof of Theorem~\ref{thm:reconstruction}]
Use the hidden coefficient representation $[w]f=\alpha^{\mathsf T}A_w\beta$
of dimension at most $S$ solely in the correctness proof.  The lemma
obtains all entries of
\[
 G_{ab}=\alpha^{\mathsf T}P_a^{\mathrm{rev}}(A)P_b(A)\beta,
 \qquad
 G^{(i)}_{ab}=\alpha^{\mathsf T}P_a^{\mathrm{rev}}(A)A_iP_b(A)\beta.
\]
Let $C=\spanop\{A_w\beta\}$ and $O=\spanop\{\alpha^{\mathsf T}A_w\}$.
The frame applied to $(A,\beta)$ makes the columns $c_b=P_b(A)\beta$
span $C$.  Applied to $(A^{\mathsf T},\alpha)$, it makes the rows
$r_a=\alpha^{\mathsf T}P_a^{\mathrm{rev}}(A)$ span $O$.
The annihilator
\[
 K_0=\{x\in C:rx=0\text{ for every }r\in O\}
\]
is invariant under every $A_i$, since $OA_i\subseteq O$.  Thus
$s=\rank G=\dim(C/K_0)$ is the minimal representation dimension:
the quotient is reachable and observable, and the word Hankel matrix
has this rank, bounding below the dimension of any other representation.

If $s=0$, return the zero polynomial, represented by an edgeless source
and sink.  Otherwise choose $I,J$ of size $s$ such that $D=G_{I,J}$
is invertible.  The classes of $c_j$, $j\in J$, form a basis of the
quotient; pairing with $r_i$, $i\in I$, gives its coordinates.  Set
\begin{equation}\label{eq:reconstructed-realization}
 \widetilde A_i=D^{-1}G^{(i)}_{I,J},\qquad
 \widetilde\beta=D^{-1}G_{I,0},\qquad
 \widetilde\alpha^{\mathsf T}=G_{0,J}.
\end{equation}
The index $0$ is essential: $P_0=1$ identifies the original seed and
output row.  These are the quotient operators and boundary vectors, so
$[w]f=\widetilde\alpha^{\mathsf T}\widetilde A_w\widetilde\beta$ for
every word.  Basis selection follows receipt of all the fixed queries.

The polynomial has degree at most $S-1$.  For $|w|\ge S$, every pairing
of $\widetilde A_w$ with an observable row and a reachable column is a
coefficient $[uwv]f=0$.  Reachability and observability of the minimal
representation imply $\widetilde A_w=0$.  Define
\[
 C_j=\sum_{|w|=j}\im\widetilde A_w,\qquad C_0=\field^s.
\]
Then $C_{j+1}\subseteq C_j$, $C_{j+1}=\sum_i\widetilde A_iC_j$, and
$C_S=0$.  Extend a basis from the deepest space upward, refining to a
complete flag on $\field^s$.  Since each operator takes $C_j$
into $C_{j+1}$, its matrix in this basis is strictly upper triangular.
Transform both boundary vectors in the same basis.  Put one internal
vertex at each of its $s$ positions, an edge $a\to b$ with label
$\sum_i(\widetilde A_i)_{ab}x_i$ when $a<b$, and constant edges from
a new source with the output-row entries and to a new sink with the
seed entries.  Its path expansion is the displayed linear representation,
including its empty-word coefficient.  It has $s+2\le S+2$ vertices.
The flag is computed by repeated images and exact basis selection,
without enumerating words.

Each of the $(m+1)Q^2$ queries returns $(2R)^2$ entries.  The matrix
units only copy query entries, and the readout vectors have the original
polynomial bit bound.  The rational ABP path-height argument applies at
order $2R$, so all answers and contractions have polynomial bit length.
Rational Gaussian elimination and inversion of $D$ retain polynomial
bit complexity.  For the flag computation, select unnormalized vectors
that are actual products of at most $S$ reconstructed operators applied
to coordinate vectors.  Their entries have polynomial bit length; the
resulting basis matrix, its inverse, and the transformed ABP coefficients
do as well.  This proves the resource claims.
\end{proof}

Finite-rank realization has classical roots in \cite{CP71}; the
field-linear Hankel quotient and minimal realization used here are
standard \cite[Chapter~2]{BR11}. Weighted-automaton learning uses word-value
and equivalence queries \cite{BBBK}. Klivans and Shpilka develop the
arithmetic-circuit learning connection \cite{KS06}. Forbes and
Shpilka give exact learning of noncommutative ABPs with membership
and equivalence queries, whose oracles can be implemented in randomized
polynomial time or deterministic quasipolynomial time
\cite[Theorem~1.5]{FSPIT}. The contribution here is deterministic
polynomial total resources with all ordinary output matrix queries fixed
before any answer. The same frame as the identity test supplies the data;
neither sparse coefficient-list output nor the intermediate-gate access
of the general-ABP reconstruction in \cite[Theorem~3.5]{AMS} is required.

\section{Rational testing with a supplied matrix base point}\label{app:basepoint}\label{sec:basepoint}

The origin-domain consequence extends to a supplied non-scalar point.
The point must belong to the domain of the given syntax, so that every
inverse gate is defined there.  Finding such a point remains a separate
task.

Realizations about arbitrary matrix centres are developed in
\cite{Volcic,PV21}. The fixed polynomial-resource query below is
relative to the supplied syntactically valid centre.

\begin{theorem}[A supplied valid matrix point]\label{thm:basepoint}
Let $m\ge1$ and work over a characteristic-zero field.  Let $q$ be a scalar-coefficient
noncommutative rational formula with $\kappa(q)\le s$, where
\[
 \kappa(q)=\#\text{constant occurrences}
       +2\#\text{letter occurrences}+\#\text{inverse occurrences}.
\]
Let $b\ge1$ and supply $P\in M_b(\field)^m$ in the syntactic domain of $q$.
From $(m,s,b,P)$ alone one can construct a single query tuple
$X$ of polynomial order that is in that domain and satisfies
$q(X)=0$ if and only if $q$ is a rational identity.
The algorithm uses output-only matrix-evaluation access and exact field
operations.  It does not obtain the formula's internal gates or search
for a domain point.
\end{theorem}
\begin{proof}
Introduce $g=mb^2$ noncommuting letters $y_{i,a,c}$ and consider, only
in the proof, the matrix-valued formal expansion
\[
 F(y)=q\left(P_i+\sum_{a,c=1}^b E_{ac}y_{i,a,c}\right)_{i=1}^m.
\]
By \cite[Corollary~5.1 and Remark~5.2]{Volcic}, the generalized
expansion of $q$ about $P$ has a realization of dimension at most $s$
over $M_b(\field)$.  Its linear transitions are finite sums of
matrix coefficients multiplying a generalized letter on both sides.
Replacing that letter by $\sum_{a,c}E_{ac}y_{i,a,c}$ makes each
transition linear in the new letters with constant matrix coefficients.
Flattening the $b\times b$ blocks therefore yields a scalar transition
representation of order at most $bs$ for each entry of $F$.  Selecting
an entry changes only its boundary vectors, not this dimension bound.

Let $T=T_{g,bs}$ be the star hitting tuple, and put
\[
 D=(g+2)bs,\qquad R=(g+2)D(D+1)/2,\qquad
 X_i=P_i\otimes I_R+\sum_{a,c=1}^bE_{ac}\otimes T_{i,a,c}.
\]
The actual query order is $bR$.  Its entries are obtained by copying
entries of $P$ and of the fixed integer matrices, with ordinary field
additions, so its construction is polynomial in $m,s,b$ and the
description of $P$.

Let $\mathcal N$ be the nonunital algebra generated by the $T$'s.
Every length-$D$ product is zero, so $\mathcal N^D=0$.  In
$M_b(\field)\otimes(\field I_R+\mathcal N)$, the subspace
$M_b(\field)\otimes\mathcal N$ is a nilpotent two-sided ideal.
Induct through the given formula: every subexpression at $X$ equals
its value at $P$ tensored with $I_R$, plus an element of this ideal.
At an inverse gate its constant matrix is invertible by the domain
promise.  Factoring it out reduces inversion to a finite geometric
sum in the nilpotent ideal.  Thus all inverse gates at $X$ are defined,
and their values agree with finite substitution in $F$.  Hence
$q(X)=F(T)$.  Theorem~\ref{thm:series}, applied to every scalar entry
with realization bound $bs$, implies
\[
 q(X)=0\quad\Longleftrightarrow\quad F=0\text{ as a matrix series}.
\]

To prove that a zero series is a rational identity, fix any $k\ge1$
and replace all new letters by independent generic $k\times k$
matrices multiplied by a central parameter $t$.  The rational entries
of the resulting evaluation have denominators nonzero at $t=0$, since
the constant tuple is $P\otimes I_k$.  Zero formal expansion makes
these rational functions zero, by injectivity of their expansion into
Laurent series in $t$.  The independent block variables parametrize an
arbitrary tuple of order $bk$; affine translation by $P\otimes I_k$
and multiplication by the nonzero indeterminate $t$ preserve that
parametrization.  Thus $q$ vanishes on its entire syntactic domain at
every order divisible by $b$.  If there were a valid nonzero evaluation
at some other order, taking a direct sum of $b$ copies would yield one
at an order divisible by $b$, a contradiction.  Conversely, a rational
identity vanishes at $X$, which we have proved to be in its domain.
\end{proof}

This is a relative exact-arithmetic theorem: it retains the supplied
valid-point promise.  It neither constructs that point nor asserts a
new bit bound for unrestricted shared rational circuits.  General
black-box rational-formula testing is the broader problem considered in
\cite{ACMRIT}.

\section*{AI methodology}
OpenAI ChatGPT and Codex helped develop the theory, including work on the
nilpotent-tag rigidity lemma (Lemma~\ref{lem:rigidity}) and related
intermediate results, and assisted with checking the proofs of most
theoretical results. They also assisted with literature review, drafting
and editing, the Lean formalization of intermediate results, and Lean
verification.
The author takes full responsibility for the content of this paper.



\begin{thebibliography}{Mathlib20}
\setlength{\itemsep}{3pt}
\raggedright
\interlinepenalty=10000

\bibitem[AK04]{AK04}
Daniel Alpay and Dmitry S. Kalyuzhny\u{\i}-Verbovetzki\u{\i}.
On the intersection of null spaces for matrix substitutions in a
non-commutative rational formal power series.
\emph{Comptes Rendus Math\'ematique}, 339(8):533--538, 2004.
\href{https://doi.org/10.1016/j.crma.2004.08.005}{doi:10.1016/j.crma.2004.08.005}.

\bibitem[ACM24]{ACMRIT}
V. Arvind, Abhranil Chatterjee, and Partha Mukhopadhyay.
Black-box identity testing of noncommutative rational formulas in
deterministic quasipolynomial time.
In \emph{Proceedings of the 56th Annual ACM Symposium on Theory of Computing},
pages 106--117, 2024.
\href{https://doi.org/10.1145/3618260.3649693}{doi:10.1145/3618260.3649693}.
Journal version: \emph{SIAM Journal on Computing}, published online
June 1, 2026,
\href{https://doi.org/10.1137/24M1689508}{doi:10.1137/24M1689508}.
Full version: \href{https://arxiv.org/abs/2309.15647}{arXiv:2309.15647}.

\bibitem[AJ26]{AJRIT}
V. Arvind and Pushkar S. Joglekar.
Rational identity testing for noncommutative circuits is in polynomial space.
\emph{arXiv preprint}, arXiv:2609.37319, 2026.
Version 2, October 2, 2026.
\href{https://arxiv.org/abs/2609.37319}{arXiv:2609.37319}.

\bibitem[AJMR19]{AJMR}
Vikraman Arvind, Pushkar S. Joglekar, Partha Mukhopadhyay, and S. Raja.
Randomized polynomial-time identity testing for noncommutative circuits.
\emph{Theory of Computing}, 15(7):1--36, 2019.
\href{https://doi.org/10.4086/toc.2019.v015a007}{doi:10.4086/toc.2019.v015a007}.

\bibitem[AMS10]{AMS}
V. Arvind, Partha Mukhopadhyay, and Srikanth Srinivasan.
New results on noncommutative and commutative polynomial identity testing.
\emph{Computational Complexity}, 19(4):521--558, 2010.
\href{https://doi.org/10.1007/s00037-010-0299-8}{doi:10.1007/s00037-010-0299-8}.
Full version: \href{https://arxiv.org/abs/0801.0514}{arXiv:0801.0514}.

\bibitem[BG19]{BG}
Moulay A. Barkatou and Renat R. Gontsov.
Linear differential systems with small coefficients:
Various types of solvability and their verification.
\emph{Symmetry, Integrability and Geometry: Methods and Applications (SIGMA)},
15:058, 15 pages, 2019.
\href{https://doi.org/10.3842/SIGMA.2019.058}{doi:10.3842/SIGMA.2019.058}.

\bibitem[BBBK00]{BBBK}
Amos Beimel, Francesco Bergadano, Nader H. Bshouty, Eyal Kushilevitz,
and Stefano Varricchio.
Learning functions represented as multiplicity automata.
\emph{Journal of the ACM}, 47(3):506--530, 2000.
\href{https://doi.org/10.1145/337244.337257}{doi:10.1145/337244.337257}.

\bibitem[BR11]{BR11}
Jean Berstel and Christophe Reutenauer.
\emph{Noncommutative Rational Series with Applications}.
Volume 137 of \emph{Encyclopedia of Mathematics and its Applications}.
Cambridge University Press, 2011.
\href{https://doi.org/10.1017/CBO9780511760860}{doi:10.1017/CBO9780511760860}.

\bibitem[BS26]{BS26}
Shubham Bhardwaj and Ramprasad Saptharishi.
Hitting sets for polynomials with small partial derivative spaces.
\emph{Electronic Colloquium on Computational Complexity},
Report TR26-217, September 28, 2026.
\url{https://eccc.weizmann.ac.il/report/2026/217/}.

\bibitem[BW05]{BW}
Andrej Bogdanov and Hoeteck Wee.
More on noncommutative polynomial identity testing.
In \emph{Proceedings of the 20th Annual IEEE Conference on Computational Complexity},
pages 92--99, 2005.
\href{https://doi.org/10.1109/CCC.2005.13}{doi:10.1109/CCC.2005.13}.

\bibitem[CP71]{CP71}
J.~W. Carlyle and A. Paz.
Realizations by stochastic finite automata.
\emph{Journal of Computer and System Sciences}, 5(1):26--40, 1971.
\href{https://doi.org/10.1016/S0022-0000(71)80005-3}{doi:10.1016/S0022-0000(71)80005-3}.

\bibitem[dMU21]{Lean4}
Leonardo de Moura and Sebastian Ullrich.
The Lean 4 theorem prover and programming language.
In \emph{Automated Deduction---CADE 28}, volume 12699 of
\emph{Lecture Notes in Computer Science}, pages 625--635. Springer, 2021.
\href{https://doi.org/10.1007/978-3-030-79876-5_37}{doi:10.1007/978-3-030-79876-5\_37}.

\bibitem[DK02]{DK}
Harm Derksen and Gregor Kemper.
\emph{Computational Invariant Theory}.
Volume 130 of \emph{Encyclopaedia of Mathematical Sciences}.
Springer, first edition, 2002.
\href{https://doi.org/10.1007/978-3-662-04958-7}{doi:10.1007/978-3-662-04958-7}.

\bibitem[DM17]{DMbounds}
Harm Derksen and Visu Makam.
Polynomial degree bounds for matrix semi-invariants.
\emph{Advances in Mathematics}, 310:44--63, 2017.
\href{https://doi.org/10.1016/j.aim.2017.01.018}{doi:10.1016/j.aim.2017.01.018}.
Full version: \href{https://arxiv.org/abs/1512.03393}{arXiv:1512.03393}.

\bibitem[DM20]{DM}
Harm Derksen and Visu Makam.
Algorithms for orbit closure separation for invariants and semi-invariants
of matrices.
\emph{Algebra \& Number Theory}, 14(10):2791--2813, 2020.
\href{https://doi.org/10.2140/ant.2020.14.2791}{doi:10.2140/ant.2020.14.2791}.
Full version: \href{https://arxiv.org/abs/1801.02043}{arXiv:1801.02043}.

\bibitem[FS13a]{FSPIT}
Michael A. Forbes and Amir Shpilka.
Quasipolynomial-time identity testing of non-commutative and
read-once oblivious algebraic branching programs.
In \emph{Proceedings of the 54th Annual IEEE Symposium on Foundations of
Computer Science (FOCS 2013)}, pages 243--252, 2013.
\href{https://doi.org/10.1109/FOCS.2013.34}{doi:10.1109/FOCS.2013.34}.
Full version: \href{https://arxiv.org/abs/1209.2408}{arXiv:1209.2408}.

\bibitem[FS13b]{FSInv}
Michael A. Forbes and Amir Shpilka.
Explicit Noether normalization for simultaneous conjugation via polynomial
identity testing.
In \emph{Approximation, Randomization, and Combinatorial Optimization.
Algorithms and Techniques (APPROX/RANDOM 2013)},
volume 8096 of \emph{Lecture Notes in Computer Science},
pages 527--542. Springer, 2013.
\href{https://doi.org/10.1007/978-3-642-40328-6_37}{doi:10.1007/978-3-642-40328-6\_37}.
Full version: \href{https://arxiv.org/abs/1303.0084}{arXiv:1303.0084}.

\bibitem[GGOW20]{GGOW}
Ankit Garg, Leonid Gurvits, Rafael Oliveira, and Avi Wigderson.
Operator scaling: Theory and applications.
\emph{Foundations of Computational Mathematics}, 20(2):223--290, 2020.
\href{https://doi.org/10.1007/s10208-019-09417-z}{doi:10.1007/s10208-019-09417-z}.

\bibitem[GV15]{GV}
Renat R. Gontsov and Ilya V. Vyugin.
Solvability of linear differential systems with small exponents in the
Liouvillian sense.
\emph{Arnold Mathematical Journal}, 1:445--471, 2015.
\href{https://doi.org/10.1007/s40598-015-0032-4}{doi:10.1007/s40598-015-0032-4}.

\bibitem[HW15]{HW}
Pavel Hrube\v{s} and Avi Wigderson.
Non-commutative arithmetic circuits with division.
\emph{Theory of Computing}, 11(14):357--393, 2015.
\href{https://doi.org/10.4086/toc.2015.v011a014}{doi:10.4086/toc.2015.v011a014}.

\bibitem[IQS18]{IQS}
G\'abor Ivanyos, Youming Qiao, and K. V. Subrahmanyam.
Constructive non-commutative rank computation is in deterministic polynomial time.
\emph{Computational Complexity}, 27(4):561--593, 2018.
\href{https://doi.org/10.1007/s00037-018-0165-7}{doi:10.1007/s00037-018-0165-7}.

\bibitem[KS06]{KS06}
Adam R. Klivans and Amir Shpilka.
Learning restricted models of arithmetic circuits.
\emph{Theory of Computing}, 2(10):185--206, 2006.
\href{https://doi.org/10.4086/toc.2006.v002a010}{doi:10.4086/toc.2006.v002a010}.

\bibitem[LM26]{LM26}
Foram Lakhani and Partha Mukhopadhyay.
Hitting point for sparse noncommutative polynomials.
\emph{Electronic Colloquium on Computational Complexity},
Report TR26-120, 2026. Revision 2, July 20, 2026.
\url{https://eccc.weizmann.ac.il/report/2026/120/}.

\bibitem[LS26]{LS26}
Foram Lakhani and Nitin Saxena.
Matrix identities are hard: Fast blackbox PIT for noncommutative
exponential-size constant-depth homogeneous circuits.
\emph{Electronic Colloquium on Computational Complexity},
Report TR26-173, September 10, 2026.
\url{https://eccc.weizmann.ac.il/report/2026/173/}.

\bibitem[Mathlib20]{Mathlib}
The mathlib Community.
The Lean mathematical library.
In \emph{Proceedings of the 9th ACM SIGPLAN International Conference
on Certified Programs and Proofs (CPP 2020)}, pages 367--381. ACM, 2020.
\href{https://doi.org/10.1145/3372885.3373824}{doi:10.1145/3372885.3373824}.

\bibitem[Mul17]{Mulmuley}
Ketan D. Mulmuley.
Geometric complexity theory V: Efficient algorithms for Noether normalization.
\emph{Journal of the American Mathematical Society}, 30(1):225--309, 2017.
\href{https://doi.org/10.1090/jams/864}{doi:10.1090/jams/864}.
Full version: \href{https://arxiv.org/abs/1209.5993}{arXiv:1209.5993}.

\bibitem[Nis91]{Nisan91}
Noam Nisan.
Lower bounds for non-commutative computation.
In \emph{Proceedings of the Twenty-Third Annual ACM Symposium on
Theory of Computing}, pages 410--418, 1991.
\href{https://doi.org/10.1145/103418.103462}{doi:10.1145/103418.103462}.

\bibitem[PV21]{PV21}
Motke Porat and Victor Vinnikov.
Realizations of non-commutative rational functions around a matrix centre,
I: Synthesis, minimal realizations and evaluation on stably finite algebras.
\emph{Journal of the London Mathematical Society}, 104(3):1250--1299, 2021.
\href{https://doi.org/10.1112/jlms.12459}{doi:10.1112/jlms.12459}.
Full version: \href{https://arxiv.org/abs/1905.11304}{arXiv:1905.11304}.

\bibitem[Pro76]{Procesi}
Claudio Procesi.
The invariant theory of $n\times n$ matrices.
\emph{Advances in Mathematics}, 19(3):306--381, 1976.
\href{https://doi.org/10.1016/0001-8708(76)90027-X}{doi:10.1016/0001-8708(76)90027-X}.

\bibitem[RS05]{RS}
Ran Raz and Amir Shpilka.
Deterministic polynomial identity testing in non-commutative models.
\emph{Computational Complexity}, 14(1):1--19, 2005.
\href{https://doi.org/10.1007/s00037-005-0188-8}{doi:10.1007/s00037-005-0188-8}.

\bibitem[Raz74]{Razmyslov}
Yu.~P. Razmyslov.
Trace identities of full matrix algebras over a field of characteristic zero.
\emph{Mathematics of the USSR-Izvestiya}, 8(4):727--760, 1974.
\href{https://doi.org/10.1070/IM1974v008n04ABEH002126}{doi:10.1070/IM1974v008n04ABEH002126}.

\bibitem[Sch61]{Schutzenberger61}
M.~P. Sch\"utzenberger.
On the definition of a family of automata.
\emph{Information and Control}, 4(2--3):245--270, 1961.
\href{https://doi.org/10.1016/S0019-9958(61)80020-X}{doi:10.1016/S0019-9958(61)80020-X}.

\bibitem[Vol18]{Volcic}
Jurij Vol\v{c}i\v{c}.
Matrix coefficient realization theory of noncommutative rational functions.
\emph{Journal of Algebra}, 499:397--437, 2018.
\href{https://doi.org/10.1016/j.jalgebra.2017.12.009}{doi:10.1016/j.jalgebra.2017.12.009}.
Full version: \href{https://arxiv.org/abs/1505.07472}{arXiv:1505.07472}.

\end{thebibliography}
\end{document}